\documentclass[a4paper]{article}
\usepackage{mathptmx}
\usepackage{amscd,amssymb,amsmath,amsthm}

\makeatletter

\renewcommand{\@seccntformat}[1]
{{\csname the#1\endcsname}.\hspace{0.3em}}

\renewcommand{\section}{\@startsection
{section}
{1}
{0mm}
{-1.5\baselineskip}
{\baselineskip}
{\bfseries\normalsize}}

\renewcommand{\subsection}{\@startsection
{subsection}
{2}
{0mm}
{-\baselineskip}
{0.5\baselineskip}
{\normalsize\itshape}}

\renewcommand{\subsubsection}{\@startsection
{subsubsection}
{3}
{0mm}
{-.5\baselineskip}
{-2mm}
{\normalsize\itshape}}

\makeatother

\theoremstyle{plain}
\newtheorem*{theorem*}{Theorem}
\newtheorem{theorem}{Theorem}[section]
\newtheorem{lemma}[theorem]{Lemma}
\newtheorem{corollary}[theorem]{Corollary}
\newtheorem{prop}[theorem]{Proposition}

\newtheorem*{corollary*}{Corollary}
\newtheorem*{DAT}{Decomposition by annuli theorem}
\newtheorem*{MT}{Mother Theorem}

\theoremstyle{definition}
\newtheorem*{defin*}{Definition}

\theoremstyle{remark}
\newtheorem*{remark*}{Remark}
\newtheorem*{example*}{Example}

\DeclareMathAlphabet{\matheur}{U}{eur}{m}{n}
\DeclareMathAlphabet{\matheus}{U}{eus}{m}{n}
\DeclareMathAlphabet{\matheuf}{U}{euf}{m}{n}

\numberwithin{equation}{section}

\newcommand{\abs}[1]{\left\lvert#1\right\rvert}

\newcommand{\norm}[1]{\left\lVert#1\right\rVert}

\DeclareMathOperator{\Div}{div}

\DeclareMathOperator{\CAP}{Cap}
\DeclareMathOperator{\re}{Re}
\DeclareMathOperator{\im}{Im}

\begin{document}

\author{Gerasim  Kokarev
\\ {\small\it Mathematisches Institut der Universit\"at M\"unchen }
\\ {\small\it Theresienstr. 39, D-80333 M\"unchen, Germany}
\\ {\small\it Email: {\tt Gerasim.Kokarev@mathematik.uni-muenchen.de}}
}

\title{Sub-Laplacian eigenvalue bounds on CR manifolds}
\date{}
\maketitle

\begin{abstract}
\noindent
We prove upper bounds for sub-Laplacian eigenvalues independent of a pseudo-Hermitian structure on a CR manifold. These bounds are compatible with the Menikoff-Sj\"ostrand asymptotic law, and can be viewed as a CR version of Korevaar's bounds for Laplace eigenvalues of conformal metrics.
\end{abstract}

\medskip
\noindent
{\small

}


\section{Introduction and statements of the results}
Let $M$ be an orientable compact strictly pseudo-convex CR manifold of dimension $(2n+1)$, possibly with boundary. For a pseudo-Hermitian structure $\theta$ on $M$ whose Levi form is positive definite we denote by
$$
0=\lambda_1(\theta)<\lambda_2(\theta)\leqslant\lambda_3(\theta)\leqslant\ldots\lambda_k(\theta)\leqslant\ldots
$$
the Neumann eigenvalues of the corresponding sub-Laplacian  $(-\Delta_b)$, see Sect.~\ref{prems} for precise definitions. Let $N_\theta(\lambda)$ be its  {\em counting function}; the value $N_\theta(\lambda)$ is the number of eigenvalues that are strictly less than a positive real number $\lambda$. By the results  of Menikoff and Sj\"ostrand~\cite{MS,Sj} and also of Fefferman and Phong~\cite{FP1,FP2} for general hypoelliptic operators, with the later improvements by Ponge~\cite{Po}, the counting function satisfies the following asymptotic formula:
\begin{equation}
\label{AL}
N_\theta(\lambda)\sim C_n\cdot\mathit{Vol}_\theta(M)\lambda^{n+1}\qquad\text{as }\quad\lambda\to+\infty,
\end{equation}
where the constant $C_n$ depends on the dimension of $M$ only. The purpose of this paper is to obtain a lower bound for $N_\theta(\lambda)$ that is compatible with the asymptotic formula. More precisely, we prove the following statement.
\begin{theorem}
\label{t1}
Let $M$ be a compact strictly pseudo-convex CR manifold, possibly with boundary. Then there exists a constant $C$, possibly depending on a CR structure, such that for any pseudo-Hermitian structure $\theta$ on $M$ with the positive definite Levi form the counting function $N_\theta(\lambda)$ satisfies the inequality
$$
N_\theta(\lambda)\geqslant C\cdot\mathit{Vol}_\theta(M)\lambda^{n+1},\qquad\text{where}\quad\lambda\geqslant 0
$$
and the volume is taken with respect to the form $\theta\wedge (d\theta)^n$.
\end{theorem}

As a consequence, this theorem yields upper bounds for the sub-Laplacian eigenvalues $\lambda_k(\theta)$ independent of a pseudo-Hermitian structure on $M$.
\begin{corollary}
\label{ei:b}
Let $M$ be a compact strictly pseudo-convex CR manifold, possibly with boundary. Then there exists a constant $C_*$, possibly depending on a CR structure, such that for any pseudo-Hermitian structure $\theta$ on $M$ with the positive definite Levi form the Neumann eigenvalues $\lambda_k(\theta)$ satisfy the inequalities
\begin{equation}
\label{KE}
\lambda_k(\theta)\mathit{Vol}_\theta(M)^{1/(n+1)}\leqslant C_*\cdot k^{1/(n+1)}\qquad\text{for any}\quad
k=1,2,\ldots,
\end{equation}
where the volume is taken with respect to the form $\theta\wedge (d\theta)^n$.
\end{corollary}
\begin{proof}
Recall that the eigenvalues are related to the counting function by the following formula
\begin{equation}
\label{eis_cnt}
\lambda_k(\theta)=\inf\{\lambda\geqslant 0: N_\theta(\lambda)\geqslant k\}.
\end{equation}
Now for a given integer $k$ let $\lambda$ be the number
$$
(C\cdot\mathit{Vol}_\theta(M))^{-1/(n+1)}\cdot k^{1/(n+1)},
$$
where $C$ is the constant from Theorem~\ref{t1}. Then by the estimate for the counting function, we obtain that $N_\theta(\lambda)\geqslant k$, and by relationship~\eqref{eis_cnt}, we conclude that $\lambda_k(\theta)\leqslant\lambda$. Setting $C_*$ to be $C^{-1/(n+1)}$, we arrive at the estimate in the statement.
\end{proof}
These eigenvalue bounds have the right asymptotic behaviour  as $k\to +\infty$, and  can be viewed as a CR version of the celebrated result by Korevaar in~\cite{Korv}, which gives Laplace eigenvalue bounds for conformal Riemannian metrics. In fact, the proof of Theorem~\ref{t1} follows a similar strategy: one of the ingredients is the result by Grigor'yan and Yau~\cite{GY99,GNY} on decompositions of a metric space by annuli, built on the original method of Korevaar. The idea is to apply their decomposition theorem to the Carnot-Caratheodory metric on $M$. We show that this choice of a metric allows to get the right capacity estimates for the annuli in the decomposition, and thus, to obtain the bound for the counting function $N_\theta(\lambda)$ compatible with the asymptotic law~\eqref{AL}.

The hypothesis above that $M$ is compact is not very essential; in Sect.~\ref{mt} we show that it can be replaced by the following two hypotheses on a background pseudo-Hermitian structure: a so-called volume doubling property together with a growth hypothesis on the volume  of Carnot-Caratheodory balls. One of the most important examples of non-compact manifolds when these hypotheses are satisfied is the Heisenberg group $\mathbf{H}_n$, see Sect.~\ref{prems}. In particular, we have the following version of Theorem~\ref{t1} for variable subdomains $\Omega\subset\mathbf{H}_n$. By $N_\theta(\Omega,\lambda)$ we denote below the number of Neumann eigenvalues $\lambda_k(\theta,\Omega)$ of a subdomain $\Omega$ that are strictly less than $\lambda$.
\begin{theorem}
\label{t2}
There exists a constant $C(n)$ such that for any compact subdomain $\Omega\subset\mathbf{H}_n$ and any pseudo-Hermitian structure $\theta$ on $\Omega$ with the positive definite Levi form the counting function $N_\theta(\Omega,\lambda)$ satisfies the inequality
$$
N_\theta(\lambda,\Omega)\geqslant C(n)\cdot\mathit{Vol}_\theta(\Omega)\lambda^{n+1},\qquad\text{where}\quad\lambda\geqslant 0
$$
and the volume is taken with respect to the form $\theta\wedge (d\theta)^n$.
\end{theorem}
An argument similar to the one in the proof of Cor.~\ref{ei:b} also yields the bounds for the Neumann eigenvalues $\lambda_k(\theta,\Omega)$ independent of a pseudo-Hermitian structure $\theta$ and a subdomain $\Omega\subset\mathbf{H}_n$. Both Theorems~\ref{t1} and~\ref{t2} are consequences of a more general statement in Sect.~\ref{mt}. In particular, a statement similar to Theorem~\ref{t2} holds for arbitrary subdomains $\Omega\subset\mathbf{H}_n$ and pseudo-Hermitian structures on $\Omega$ with finite volume. In fact, following the approach of Grigor'yan and Yau~\cite{GY99,GNY} in the conformal geometry setting, one can give versions of these statements for rather arbitrary measures on CR manifolds.

The method used in this paper carries over essentially without changes to contact sub-Riemannian manifolds and can be also used to get sub-Laplacian eigenvalue bounds on general sub-Riemannian manifolds. We plan to address these and related questions in a forthcoming paper. 

\noindent
{\em Acknowledgements.} I am grateful to Andrei Agrachev for a number of discussions on the subject during my visit to SISSA, Trieste.

\section{Preliminaries}
\label{prems}
\subsection{Basic notions of CR geometry}
We start with recalling a basic material on CR manifolds; for the details we refer to~\cite{DT}. Let $M$ be a smooth manifold of dimension $2n+1$; we always assume that it is connected and orientable. A {\em CR structure} of type $(n,1)$ on $M$ is a complex sub-bundle $H^{1,0}$ of the complexified tangent bundle $TM\otimes\mathbf{C}$ that satisfies the following properties:
\begin{itemize}
\item[(i)] its complex rank is equal to $n$, and $H^{1,0}_x\cap\overline{H_x^{1,0}}=\{0\}$ for any $x\in M$;
\item[(ii)] it is formally integrable: $[H^{1,0},H^{1,0}]\subset H^{1,0}$, that is the Lie bracket of any $H^{1,0}$-valued vector field is an $H^{1,0}$-valued vector field.
\end{itemize}
A manifold equipped with such a CR structure is referred to as a {\em CR manifold.}

The sum $\re (H^{1,0}\oplus\overline{H^{1,0}})$ is called the {\em Levi distribution} on $M$. It is a real rank $2n$ sub-bundle of $TM$, denoted by $H$, which carries a natural complex structure
$$
J_b:(v+\bar v)\longmapsto i(v-\bar v).
$$
A {\em pseudo-Hermitian structure} $\theta$ on $M$ is a nowhere vanishing $1$-form whose kernel coincides with $H$. On an orientable manifold $M$ any two pseudo-Hermitian structures $\hat\theta$ and $\theta$ are related as $\hat\theta=\kappa\cdot\theta$, where $\kappa$ is a nowhere vanishing function. The {\em Levi form} $L_\theta$ of a pseudo-Hermitian structure $\theta$ is defined as
$$
L_\theta(v,\bar w)=-id\theta(v,\bar w),\qquad\text{where~ }v,w\in H^{1,0}.
$$
It is straightforward to see that the  Levi forms of two pseudo-Hermitian structures $\hat\theta$ and $\theta$ are related as $L_{\hat\theta}=\kappa\cdot L_\theta$; this is an important relation for the sequel. A CR manifold is called {\em strictly pseudo-convex} if the Levi form $L_\theta$ is positive definite for some pseudo-Hermitian structure $\theta$.

Let $M$ be a pseudo-convex CR manifold equipped with a pseudo-Hermitian structure $\theta$, and $\pi_H$ be the projection $TM\to H$ along the $1$-dimensional kernel of $d\theta$. Then the bilinear form
$$
d\theta(\pi_H\cdot,J_b\pi_H\cdot)+(\theta\otimes\theta)(\cdot,\cdot)
$$
is a Riemannian metric on $M$, referred to as the {\em Webster metric}. A straightforward calculation shows that its volume form coincides with $\theta\wedge(d\theta)^n$ up to a constant factor, which depends on $n$ only, see~\cite{DT}.

Basic examples of CR manifolds include real hypersurfaces in ${\mathbf C}^{n+1}$; for example, any odd-dimensional unit sphere $S^{2n+1}\subset\mathbf{C}^{n+1}$ is a strictly pseudo-convex CR manifold, see~\cite{DT}. Now we describe a more special example.
\begin{example*}[The Heisenberg group]
The space $\mathbf{C}^n\times\mathbf{R}$ viewed as a group with respect to the following operation
$$
(z,t)\cdot(w,s)=(z+w,t+s+2\im\langle z,w\rangle),
$$
where $\langle\cdot,\cdot\rangle$ is a standard Hermitian product in $\mathbf{C}^n$, is called the Heisenberg group $\mathbf{H}_n$. Its CR structure is given by the span of the vectors
\begin{equation}
\label{cr:fields}
T_i=\frac{\partial~}{\partial z_i}+i\bar{z}^i\frac{\partial~}{\partial t},\qquad\text{where~ } i=1,\ldots,n.
\end{equation}
As a pseudo-Hermitian structure one can take the $1$-form
$$
\theta=dt+i\sum^{n}_{j=1}\left(z^jd\bar{z}^j-\bar{z}^jdz^j\right);
$$
its differential  is $2i\sum dz^j\wedge d\bar{z}^j$. Thus, we see that the corresponding Levi form is positive definite, and $\mathbf{H}_n$ is a strictly pseudo-convex CR manifold. For any $x=(z,t)\in\mathbf{H}_n$ the Heisenberg norm $\abs{x}$ is defined as $(\norm{z}^4+t^2)^{1/4}$, where $\norm{\cdot}$ is a Euclidean norm in $\mathbf{C}^n$. The Heisenberg norm has a nice property of being homogeneous under the dilatations
$$
D_\epsilon(z,t)=(\epsilon z,\epsilon^2t),\qquad\text{where~ }(z,t)\in\mathbf{H}_n,\text{ ~and~ }\epsilon>0, 
$$
which are group and CR isomorphisms, see~\cite{DT}. It is also closely related to the Carnot-Caratheodory distance on $\mathbf{H}_n$, which we recall below.
\end{example*}

\subsection{Carnot-Caratheodory distance on CR manifolds}
Let $M$ be a strictly pseudo-convex CR manifold and $H$ be its Levi distribution. It is a straightforward exercise to show that for any point $x\in M$ and any local frame $\{X_i\}$ of the sub-bundle $H$ around $x$ the commutators of the $X_i$'s at $x$ span the  tangent space $T_xM$. This property is often referred to as the {\em H\"ormander condition} for the distribution $H$.

An absolutely continuous path $\gamma$ in $M$ is called {\em horizontal} if it is tangent to $H$ almost everywhere. For a pseudo-Hermitian structure on $M$ whose Levi form is positive definite the {\em Carnot-Caratheodory distance} between the points $x$ and $y$ is defined as
$$
d_\theta(x,y)=\inf\{\mathit{length_\theta(\gamma)} : \gamma\text{ is a horizontal path joining }x\text{ and }y\},
$$
where the length is taken with respect to the Webster metric. By the results of Caratheodory and Chow, see~\cite{NSW}, the H\"ormander condition implies that this metric is finite. Besides, it induces the same topology on $M$. 

The following statement is  a version of~\cite[Th.~1]{NSW}, see also~\cite{Mitch}; it is a consequence of the so-called "ball-box property".
\begin{prop}
\label{ess}
Let $M$ be a pseudo-convex CR manifold and $\theta$ be a pseudo-Hermitian structure whose Levi form is positive definite. Then for every compact set $K\subset M$ there exist positive constants $c_1$, $c_2$, and $\rho$, and a degree $(2n+2)$ polynomial $P_{2n+2}(x,r)$, whose coefficients are smooth positive functions in $x$, such that
$$
c_1\cdot P_{2n+2}(x,r)\leqslant\mathit{Vol}_\theta(B(x,r))\leqslant c_2\cdot P_{2n+2}(x,r)
$$
for an arbitrary Carnot-Caratheodory ball centred at $x\in K$ with $r\leqslant\rho$. 
\end{prop}
As a consequence, we arrive at the following corollary.
\begin{corollary}
\label{pts}
Let $M$ be a compact strictly pseudo-convex CR manifold. Then the volume measure of any pseudo-Hermitian structure $\theta$ whose Levi form is positive definite satisfies the following properties:
\begin{itemize}
\item[(i)] the {\em doubling property}
$$
\mathit{Vol}_{\theta}(B(x,2r))\leqslant C_1\cdot\mathit{Vol}_{\theta}(B(x,r))
$$
for Carnot-Caratheodory balls with arbitrary $x$ and $r>0$, and some $C_1>0$;
\item[(ii)] the {\em growth property}
$$
\mathit{Vol}_{\theta}(B(x,r))\leqslant C_2\cdot r^{2(n+1)}
$$
for arbitrary Carnot-Caratheodory balls in $M$ and some constant $C_2>0$.
\end{itemize}
\end{corollary}
\begin{proof}[Sketch of the proof]
The first statement is equivalent to saying that the quotient
$$
\mathit{Vol}_\theta(B(x,2r))/\mathit{Vol}_\theta(B(x,r))
$$
is bounded in $r>0$ and $x\in M$. By Prop.~\ref{ess} it is so for any $r\leqslant\rho/2$ and $x\in M$. Now since $M$ is compact, its volume is finite. By Prop.~\ref{ess}, we see that the denominator is bounded from below for $r>\rho/2$, and conclude that this quotient is bounded from above for any $r>0$. The growth property~$(ii)$ follows in a similar fashion.
\end{proof}
\begin{example*}[continued]
The Levi distribution on the Heisenberg group $\mathbf{H}_n$ is the span of the vector fields
$$
X_i=\frac{\partial~}{\partial x^i}+2y^i\frac{\partial~}{\partial t},\qquad
Y_i=\frac{\partial~}{\partial y^i}-2x^i\frac{\partial~}{\partial t},
$$
where $z^i=x^i+\sqrt{-1}y^i$ and $(z^1,\ldots,z^n,t)\in\mathbf{H}_n$. These vector fields are twice the real and (minus) imaginary parts of the left-invariant vector fields $T_i$'s given by~\eqref{cr:fields}. Besides, the Levi form of the standard pseudo-Hermitian structure is also left-invariant, and so is the corresponding Carnot-Caratheodory metric $d_\theta(\cdot,\cdot)$ on $\mathbf{H}_n$, see~\cite{DT}. It is straightforward to show that the latter commutes with dilatations
$$
d_\theta(D_\varepsilon x,D_\varepsilon y)=\varepsilon\cdot d_\theta(x,y)\qquad\text{for any }
x,y\in\mathbf{H}_n\text{ and }\varepsilon>0.
$$
Thus, we conclude that $\mathbf{H}_n$ is a non-compact CR manifold that satisfies both properties~$(i)$ and~$(ii)$ in Cor.~\ref{pts}. There is also another natural metric on the Heisenberg group:
$$
d(x,y)=\abs{x^{-1}y},\qquad\text{where }x, y\in\mathbf{H}_n,
$$
and $\abs{\,\cdot\,}$ is the Heisenberg norm. By definition this metric is left-invariant, and since the Heisenberg norm is homogeneous, it also commutes with dilatations. Any metric with the last two properties is determined by its values on $\{0\}\times\partial B(0,1)$. From this we conclude that such a metric has to be equivalent to the Carnot-Caratheodory metric $d_\theta(\cdot,\cdot)$.
\end{example*}

\subsection{Sub-Laplacian on CR manifolds}
Let $M$ be a pseudo-convex CR manifold equipped with a pseudo-Hermitian structure $\theta$ whose Levi form is positive definite; by $\pi_H$ we denote the projection $TM\to H$ along the $1$-dimensional kernel of $d\theta$. The {\em sub-Laplacian} $(-\Delta_b)$ is a second order differential operator defined as 
$$
-\Delta_bu=-\Div(\nabla_bu),\qquad\text{where}\quad\nabla_bu=\pi_H\nabla u,
$$
and $\nabla u$ is regarded as a vector field with respect to the Webster metric; the divergence is understood in the sense of the volume form $\theta\wedge(d\theta)^n$. Equivalently, it can be defined as the Hormander operator
$$
-\Delta_bu=\sum^{2n}_{i=1} X^*_iX_iu,
$$
where $(X_i)$ is a local orthonormal frame of the sub-bundle $H$ with respect to the Webster metric, and the $(X^*_i)$'s are  the adjoint operators with respect the natural $L_2$-scalar product. 
\begin{example*}[continued]
The sub-Laplacian on the Heisenberg group $\mathbf{H}_n$ has the form
$$
-\Delta_b=\frac{1}{2}\sum_{i=1}^n\left(\frac{\partial^2~}{\partial z^i\partial\bar z^i}+\abs{z}^2\frac{\partial^2}{\partial t^2}-i\frac{\partial}{\partial t}\left(z^i\frac{\partial~}{\partial z^i}-\bar z^i\frac{\partial~}{\partial\bar z^i}\right)\right),
$$
where $(z^1,\ldots,z^n,t)$ are standard coordinates on $\mathbf{C}^n\times\mathbf{R}$, see~\cite{DT}.
\end{example*}

As is known, the sub-Laplacian is a sub-elliptic operator of order $1/2$. This means that for any $C^\infty$-smooth compactly supported function $u$ the following inequality holds
$$
{\vert\!\!\vert u\vert\!\!\vert}_{1/2}^2\leqslant C\left({\vert\!\!\vert}{(\Delta_bu,u)}{\vert\!\!\vert}+{{\vert\!\!\vert}{u}{\vert\!\!\vert}}^2\right)
$$
for some constant $C$, where ${\vert\!\!\vert\,\cdot\,\vert\!\!\vert}$ and ${\vert\!\!\vert\,\cdot\,\vert\!\!\vert}_{1/2}$ stand for the $L_2$-norm and the Sobolev $(1/2)$-norm respectively, see~\cite{DT}. On a compact manifold $M$ we view $(-\Delta_bu)$ as an operator defined on $C^\infty$-smooth functions that satisfy the following version of the {\em Neumann boundary hypothesis}:
\begin{equation}
\label{neumann}
\langle\nabla_bu,\vec{n}\rangle=0\text{ on }\partial M,\qquad\text{where }\vec{n}\text{ is an outward normal to }\partial M,
\end{equation}
and the brackets denote the scalar product in the sense of the Webster metric. By the CR version of Green's formula
$$
\int_M (\Delta_bu)v\mathit{dVol}_\theta+\int_M\langle\nabla_bu,\nabla_bv\rangle\mathit{dVol}_\theta=
\int_{\partial M}\langle\nabla_bu,\vec{n}\rangle v\mathit{dS}_\theta,
$$
we see that this operator is formally self-adjoint. In particular, its resolvent set is not empty. Using the compact Sobolev embedding together with the sub-ellipticity hypothesis, it is straightforward to conclude that its resolvent is compact and the spectrum is discrete. By the same Green's formula, it coincides with the spectrum of the Dirichlet form $\int\abs{\nabla_b u}^2\mathit{dVol}_\theta$, viewed as a form on smooth functions on $M$.

If a manifold $M$ is non-compact, then we assume that the sub-Laplacian is defined on $C^\infty$-smooth functions that, in addition to the Neumann boundary hypothesis~\eqref{neumann}, are compactly supported. As is known, such an operator is closable in the space of $L_2$-integrable functions, and its counting function $N_\theta(\lambda)$ can be defined as
$$
N_\theta(\lambda)=\sup\left\{\dim V: \int\abs{\nabla_b u}^2\mathit{dVol}_\theta<\lambda\int u^2\mathit{dVol}_\theta\text{ for any }u\in V\backslash\{0\}\right\},
$$
where $V$ is  a subspace formed by smooth functions. We assume that the supremum over the empty set equals zero. The eigenvalues of such an operator are precisely the jump points of the counting function,
$$
\lambda_k(\theta)=\inf\left\{\lambda\geqslant 0: N_\theta(\lambda)\geqslant k\right\}.
$$
We refer to~\cite{GNY}, where the details on a more general setting can be found.

\section{General statement}
\label{mt}

Let $M$ be a strictly pseudo-convex CR manifold. In sequel by a pseudo-Hermitian structure on $M$ we mean a pseudo-Hermitian structure with the positive definite Levi form. By $\lfloor x\rfloor$ we also denote the floor function of $x\in\mathbf{R}$, the greatest integer that is at most $x$.  The following general statement is proved in Sect.~\ref{proofs}.
\begin{MT}
Let $M$ be a strictly pseudo-convex CR manifold of dimension $2n+1$, possibly with boundary. Let $\theta_0$ be a background pseudo-Hermitian structure whose  Carnot-Caratheodory metric is complete and the volume measure $\mathit{Vol}_{\theta_0}$ satisfies the following properties:
\begin{itemize}
\item[(i)] the {\em doubling property}
$$
\mathit{Vol}_{\theta_0}(B(x,2r))\leqslant C_1\cdot\mathit{Vol}_{\theta_0}(B(x,r))
$$
for Carnot-Caratheodory balls with arbitrary $x$ and $r>0$, and some $C_1>0$;
\item[(ii)] the {\em growth property}
$$
\mathit{Vol}_{\theta_0}(B(x,r))\leqslant C_2\cdot r^{2(n+1)}
$$
for arbitrary Carnot-Caratheodory balls in $M$ and some constant $C_2>0$.
\end{itemize}
Then there exists a constant $C$, depending on $C_1$, $C_2$, and $n$ only, such that for any  pseudo-Hermitian structure $\theta$ of finite total volume, $\mathit{Vol}_\theta(M)<+\infty$, the counting function $N_\theta(\lambda)$ satisfies the inequality 
$$
N_\theta(\lambda)\geqslant\lfloor C\cdot\mathit{Vol}_\theta(M)\lambda^{n+1}\rfloor\qquad\text{for any~}\lambda\geqslant 0.
$$
In particular, the sub-Laplacian eigenvalues $\lambda_k(\theta)$ satisfy the following inequalities
$$
\lambda_k(\theta)\mathit{Vol}_\theta(M)^{1/(n+1)}\leqslant C_*\cdot k^{1/(n+1)}\qquad\text{for any~} k=1,2,\ldots,
$$
where $C_*=C^{-1/(n+1)}$.
\end{MT}
First, if the manifold $M$ is compact, then, by Cor.~\ref{pts}, the properties $(i)$ and $(ii)$ hold for an arbitrary pseudo-Hermitian structure. In this case the above estimate for the counting function $N_\theta(\lambda)$ can be improved -- no floor function is necessary. Indeed, when $M$ is compact, we always have an extra test-function, the constant, which yields an improved estimate, see Sect.~\ref{proofs}. With this improvement, Theorem~\ref{t1} becomes  a consequence of the Mother Theorem. Further, the standard pseudo-Hermitian structure on the Heisenberg group $\mathbf{H}_n$ satisfies the hypotheses~$(i)$ and~$(ii)$, and so does any its subdomain $\Omega\subset\mathbf{H}_n$ with the same constants $C_1$ and $C_2$. Since the constant $C$ in the estimate for the counting function depends only on $C_1$, $C_2$, and $n$, we also obtain the statement in Theorem~\ref{t2}.

Finally, mention that, following the idea of Grigor'yan and Yau in~\cite[Sect.~5.3]{GNY} in the conformal setting, the statement above can be generalised even further in another direction. Let $\sigma$ be a non-atomic finite Radon measure on $M$, and $N_\theta(\lambda,\sigma)$ be a counting function corresponding to the sub-Laplacian viewed as an operator in $L_2(M,\sigma)$. Then a straightforward generalisation of the argument in Sect.~\ref{proofs} shows that for any pseudo-Hermitian structure $\theta$ of finite total volume, $\mathit{Vol}_\theta(M)<+\infty$, the counting function $N_\theta(\lambda,\sigma)$ satisfies the inequality 
$$
N_\theta(\lambda,\sigma)\geqslant\lfloor C\cdot\frac{\sigma(M)^{n+1}}{\mathit{Vol}_\theta(M)^n}\lambda^{n+1}\rfloor\qquad\text{for any~}\lambda\geqslant 0,
$$
where the constant $C$ depends on $C_1$, $C_2$, and $n$ only. In particular, if $\sigma$ coincides with the volume measure $\mathit{Vol}_\theta$, then we obtain the original estimate in the Mother Theorem. As a consequence, we also have the following bounds for the eigenvalues $\lambda_k(\theta,\sigma)$:
$$
\lambda_k(\theta,\sigma)\frac{\sigma(M)}{\mathit{Vol}_\theta(M)^{n/(n+1)}}\leqslant C_*\cdot k^{1/(n+1)}\qquad\text{for any~} k=1,2,\ldots,
$$
where $C_*=C^{-1/(n+1)}$. Such considerations can be useful when one deals with other boundary conditions on $M$, see~\cite[Example~1.3]{Ko} and~\cite{CEG}.

\section{The proof}
\label{proofs}
\subsection{Prerequisities~I: metric space decomposition theorem}
We start with recalling the general result of Grigor'yan and Yau~\cite{GY99,GNY} concerning the decomposition of a metric space $(X,d)$ by annuli. By an annulus $A$ in  $X$ we mean a subset of the following form
$$
\{x\in X: r\leqslant d(x,a)<R\},
$$
where $a\in X$ and $0\leqslant r<R<\infty$. We also use the notation $2A$ for the annulus
$$
\{x\in M:r/2\leqslant d(x,a)<2R\}.
$$ 
Recall that a measure $\mu_0$ on $M$ is called {\em doubling}, if 
\begin{equation}
\label{double}
\mu_0(B(x,2r))\leqslant C_1\cdot\mu_0(B(x,r))
\end{equation}
for arbitrary $x\in X$ and $r>0$. We also say that a given measure $\mu$ on $X$ is {\em non-atomic}, if the mass of every point is zero.

Building on the ideas of Korevaar~\cite{Korv}, Grigor'yan and Yau showed that on certain metric spaces for any non-atomic measure $\mu$,  one can always find a collection of disjoint annuli $\{2A_i\}$ such that the values $\mu(A_i)$ are bounded below by some positive constant. More precisely, in~\cite{GY99,GNY} they prove the following statement.
\begin{DAT}
Let $(X,d)$ be a separable metric space whose all metric balls are precompact and that admits a doubling measure $\mu_0$. Then for any finite non-atomic measure $\mu$ and any positive integer $k$ there exists a collection $\{2A_i\}$ of $k$ disjoint  annuli such that
$$
\mu(A_i)\geqslant c\mu(X)/k\qquad\text{for any~}1\leqslant i\leqslant k,
$$
where the positive constant $c$ depends only on the doubling constant $C_1$ of $\mu_0$.
\end{DAT}
Mention that  in~\cite{GNY} this theorem is stated for metric spaces that satisfy the {\em metric doubling hypothesis:} there exists a constant $N$ such that any metric ball of radius $r$ can be covered by at most $N$ balls of radii $r/2$. For such a space no hypothesis on the existence of a doubling measure is necessary. However, by standard covering theorems~\cite{Matt}, it is straightforward to see that any separable metric space that carries a measure with doubling property~\eqref{double} satisfies the metric doubling hypothesis with the constant $N$ that depends only on the doubling constant $C_1$.

We are going to apply this theorem to a CR manifold $M$ viewed as a metric space with the Carnot-Caratheodory distance. In this setting, by the Hopf-Rinow theorem for length spaces~\cite{BBI}, the hypothesis that all metric balls are precompact is equivalent to the completeness of $M$. This is clearly satisfied when $M$ is compact, or has the structure of a Carnot group, or is complete in the Webster metric.

The use of annuli in the decomposition theorem is actually important. It allows to estimate the capacities $\CAP(A,2A)$, knowing an estimate for the capacities $\CAP(B,2B)$ of open metric balls $B$, see Lemma~\ref{cap2}. For our purposes it is appropriate to use a notion of capacity based on the CR version of the Dirichlet energy, that is the energy $\int\abs{\nabla_b u}^2\mathit{dVol}_\theta$. We describe it below in more detail.

\subsection{Prerequisites~II: capacities on CR manifolds}
Let $M$ be a strictly pseudo-convex CR manifold of dimension $(2n+1)$, and $\theta$ be a pseudo-Hermitian structure whose Levi form is positive definite. For a compact subdomain $\Omega\subset M$ and a real number $1\leqslant p<+\infty$ by $W^{1,p}(\Omega)$ we denote the Sobolev space formed by real-valued functions $u$ whose $(1,p)$-norm
\begin{equation}
\label{norm}
\abs{u}_{1,p}=\left(\int_\Omega\abs{u}^p\mathit{dVol}_\theta\right)^{1/p}+\left(\int_\Omega\abs{\nabla_bu}^p\mathit{dVol}_\theta\right)^{1/p}
\end{equation}
is finite. Recall that a real-valued function $u$ on $M$ is called Lipschitz in the sense of the Carnot-Caratheodory metric $d_\theta$, if there exists a constant $L$  such that
\begin{equation}
\label{Lcc}
\abs{u(x)-u(y)}\leqslant L\cdot d_\theta(x,y)\qquad\text{ for any }x\text{ and }y\text{ in }M.
\end{equation}
By the results in~\cite{NSW}, any smooth function is locally Lipschitz in the above sense. We proceed with  the following proposition, which is a reformulation of the results due to~\cite{FSSC96,FSSC97} and~\cite{GN98}.
\begin{prop}
\label{lip}
Let $M$ be a strictly pseudo-convex CR manifold, and $\theta$ be a pseudo-Hermitian structure whose Levi form is positive definite. 
\begin{itemize}
\item[(i)] Let $u$ be a Lipschitz function on $M$ in the sense of the Carnot-Caratheodory distance $d_\theta$, that is, $u$ satisfies~\eqref{Lcc}. Then the distributional derivative $\nabla_bu$ is a measurable bounded function whose absolute value $\abs{\nabla_bu}$ is not greater than $L$. 
\item[(ii)] The space $W^{1,p}(\Omega)$ is the closure in the norm~\eqref{norm} of the space of smooth functions on a compact subdomain $\Omega\subset M$.
\end{itemize}
\end{prop}
Now we describe the concept of capacity on CR manifolds. Recall that the {\em capacitor} in $M$ is a pair $(F,G)$ of Borel sets $F\subset G$. For a given pseudo-Hermitian structure $\theta$ and a real number $1\leqslant p<+\infty$ the {\em sub-Riemannian $p$-capacity} of a capacitor $(F,G)$ is defined as
$$
\CAP_p(F,G)=\inf\left\{\int_M\abs{\nabla_bu}^p\mathit{dVol}_\theta\right\}.
$$
Here the infimum is taken over all Lipschitz (in the Carnot-Caratheodory sense) test-functions on $M$, that is the functions that are equal to $1$ in a neighbourhood of $F$ and whose support is compact and lies in the interior of $G$. Using Prop.~\ref{lip} and the mollification technique, it is straightforward to show that the $p$-capacity $\CAP_p(F,G)$ can be equivalently defined by using smooth test-functions, see~\cite{GN98}.

In sequel we use the following auxiliary lemmas. 
\begin{lemma}
\label{cap1}
Let $M$ be a strictly pseudo-convex CR manifold equipped with a pseudo-Hermitian structure $\theta$ whose Levi form is positive definite. Suppose that $M$ is complete in the Carnot-Caratheodory metric $d_\theta(\cdot,\cdot)$. Then the sub-Riemannian $p$-capacity  of Carnot-Caratheodory balls satisfies the following inequality
$$
\CAP_p(B(x,r),B(x,2r))\leqslant \mathit{Vol}_\theta(B(x,2r))\cdot r^{-p}
$$
for arbitrary $x\in M$ and $r>0$.
\end{lemma}
\begin{proof}
For a given $r>0$ choose $\varepsilon$ such that $0<\varepsilon<r$. Consider the function $u_\varepsilon$ defined as
$$
u_\varepsilon(y)=\left\{
\begin{array}{ll}
1 & \text{for }\quad d_\theta(x,y)\leqslant r+\varepsilon/2;\\
\frac{\displaystyle 2r-\varepsilon/2-d_\theta(x,y)}{\displaystyle r-\varepsilon} & \text{for }\quad r+\varepsilon/2<d_\theta(x,y)<2r-\varepsilon/2;\\
0 & \text{for }\quad d_\theta(x,y)\geqslant 2r-\varepsilon/2.
\end{array}
\right.
$$
Clearly, it is Lipschitz and  can be used as a test-function for the $p$-capacity. Besides, it is straightforward to see that
$$
\abs{\nabla_bu_\varepsilon}\leqslant (r-\varepsilon)^{-1},
$$
and thus, we obtain 
$$
\CAP_p(B(x,r),B(x,2r))\leqslant \mathit{Vol}_\theta(B(x,2r))\cdot(r-\varepsilon)^{-p}.
$$
Passing to the limit as $\varepsilon\to 0+$, we obtain the inequality in the statement of the lemma.
\end{proof}
\begin{lemma}
\label{cap2}
Let $M$ be a strictly pseudo-convex CR manifold equipped with a pseudo-Hermitian structure $\theta$ whose Levi form is positive definite. Then for four nested Borel sets $F\subset G\subset F'\subset G'$ the following inequality 
$$
\CAP_p(F'\backslash G, G'\backslash F)^{1/p}\leqslant\CAP_p(F,G)^{1/p}+\CAP_p(F',G')^{1/p}
$$
holds, where $1\leqslant p<+\infty$.
\end{lemma}
\begin{proof}
If $u$ and $v$ are test-functions for the capacitors $(F,G)$ and $(F',G')$ respectively, then the function $(v-u)$ is a test-function for the capacitor $(F'\backslash G,G'\backslash F)$. Now by the Minkowski inequality, we obtain
\begin{multline}
\left(\int\abs{\nabla_b(v-u)}^p\mathit{dVol}_\theta\right)^{1/p}\leqslant \left(\int(\abs{\nabla_bu}+\abs{\nabla_bv})^p\mathit{dVol}_\theta\right)^{1/p}\\
\leqslant\left(\int\abs{\nabla_bu}^p\mathit{dVol}_\theta\right)^{1/p}+\left(\int\abs{\nabla_bv}^p\mathit{dVol}_\theta\right)^{1/p}.
\end{multline}
Taking the infimums over test-functions, we arrive at the inequality in the lemma.
\end{proof}
Mention that various capacity estimates for concentric Carnot-Caratheodory balls have been obtained by a number of authors in the past, see~\cite{GM10} and the reference there. However, many of them have a form different from the one in Lemma~\ref{cap1} and have been obtained in view of different applications.

\subsection{Proof of the Mother Theorem}
For a given pseudo-Hermitian structure $\theta$ with the positive definite Levi form consider the $(2n+2)$-capacity
$$
\CAP_{2n+2}(F,G)=\inf\left\{\int_M\abs{\nabla_bu}^{2n+2}\mathit{dVol}_\theta\right\},
$$
where the infimum is taken over smooth (or Lipschitz) test-functions. Recall that any pseudo-Hermitian structure $\hat\theta$ with the positive definite Levi form has the form $\hat\theta=\kappa\cdot\theta$, where $\kappa$ is a positive smooth function. Using the relationship $L_{\hat\theta}=\kappa\cdot L_\theta$ for the Levi forms together with the formula  for  the volume form $\mathit{dVol}_\theta$, we conclude that this capacity is invariant with respect to the change of a pseudo-Hermitian structure. In particular, we have
$$
\CAP_{2n+2}(F,G)=\inf\left\{\int_M\abs{\nabla_bu}^{2n+2}\mathit{dVol}_{\theta_0}\right\},
$$
where $\theta_0$ is a background pseudo-Hermitian structure that satisfies the hypotheses~$(i)$ and~$(ii)$ in the theorem. Combining Lemma~\ref{cap1} with the hypothesis~$(ii)$, we get
$$
\CAP_{2n+2}(B,2B)\leqslant 2^{2n+2}C_2,
$$
for any Carnot-Caratheodory ball $B$, where $2B$ denotes the ball centered at the same point whose radius is twice greater. By Lemma~\ref{cap2}, we also obtain the bound
$$
\CAP_{2n+2}(A,2A)\leqslant c(n,C_2)
$$
for capacitors formed by annuli. Using the H\"older inequality we further get
\begin{equation}
\label{eq1}
\CAP_{2}(A,2A)\leqslant \bar c(n,C_2)\mathit{Vol}_\theta(2A)^{n/(n+1)}.
\end{equation}
Now let $c(C_1)$ be the constant from the annuli decomposition theorem for the space $(M,d_{\theta_0})$, and  set
$$
C=(4\bar c(n,C_2)/c(C_1))^{-(n+1)}.
$$
For a given $\lambda>0$ we denote by $k$  the integer $\lfloor C\cdot\mathit{Vol}_\theta(M)\lambda^{n+1}\rfloor$. By the annuli decomposition theorem there exists a collection $\{A_i\}$ of $2k$ Carnot-Caratheodory annuli such that
\begin{equation}
\label{eq2} 
\mathit{Vol}_\theta(A_i)\geqslant c(C_1)\mathit{Vol}_\theta(M)/(2k)\qquad\text{for every~}i=1,\ldots,2k,
\end{equation}
and the annuli $\{2A_i\}$ are disjoint. The latter implies that
$$
\sum_{i=1}^{2k}\mathit{Vol}_\theta(2A_i)\leqslant\mathit{Vol}_\theta(M),
$$
and hence, there exists at least $k$ sets $2A_i$ such that
$$
\mathit{Vol}_\theta(2A_i)\leqslant\mathit{Vol}_\theta(M)/k.
$$
Without loss of generality, we can suppose that these inequalities hold for $i=1,\ldots,k$. Combining this with relationship~\eqref{eq1}, we obtain
$$
\CAP_{2}(A_i,2A_i)\leqslant\bar c(n,C_2)\left(\mathit{Vol}_\theta(M)/k\right)^{n/(n+1)}
$$ 
for any $i=1,\ldots,k$. To demonstrate the lower bound for the counting function it is sufficient to construct $k$ linearly independent functions $\{u_i\}$ such that
$$
\int\abs{\nabla_bu_i}^2\mathit{dVol}_\theta<\lambda\cdot\int u_i^2\mathit{dVol}_\theta.
$$
As such functions one can take nearly optimal test-functions for the capacitors $(A_i,2A_i)$, where $i=1,\ldots,k$. In more detail, for a sufficiently small $\varepsilon$ we can choose test-functions $u_i$ such that
$$
\int\abs{\nabla_bu_i}^2\mathit{dVol}_\theta<\CAP_{2}(A_i,2A_i)+\varepsilon\leqslant 2\bar c(n,C_2)\left(\mathit{Vol}_\theta(M)/k\right)^{n/(n+1)}.
$$
Now using inequality~\eqref{eq2}, we get
$$
\left(\int\abs{\nabla_bu_i}^2\mathit{dVol}_\theta\right)/\left(\int u_i^2\mathit{dVol}_\theta\right)<\frac{4\bar c(n,C_2)}{c(C_1)}(k/\mathit{Vol}_\theta(M))^{1/(n+1)}\leqslant\lambda,
$$
where the last inequality follows by the choice of the quantities $C$ and $k$. This proves the lower bound for the counting function. The upper bounds for the eigenvalues can be now obtained in a fashion similar to the one in the proof of Cor.~\ref{ei:b}.


{\small
\begin{thebibliography}{99}
\addcontentsline{toc}{section}{References}


\bibitem{BBI} Burago,~D., Burago,~Y., Ivanov,~S. {\em A course in metric geometry.} Graduate Studies in Mathematics, {\bf 33}. AMS, Providence, RI, 2001. xiv+415 pp.

\bibitem{CEG} Colbois,~B., El Soufi,~A., Girouard,~A. {\em Isoperimetric control of the Steklov spectrum.} J. Funct. Anal. {\bf 261} (2011), 1384--1399.

\bibitem{DT} Dragomir,~S., Tomassini,~G. {\em Differential geometry and analysis on CR manifolds.} Progress in Mathematics, 246. Birkh\"auser Boston, Inc., Boston, MA, 2006. xvi+487

\bibitem{FP1} Fefferman,~C., Phong,~D.~H. {\em On the asymptotic eigenvalue distribution of a pseudo-differential operator.} Proc. Nat. Acad. Sci. {\bf 77} (1980), 5622--5625.

\bibitem{FP2} Fefferman,~C., Phong,~D.~H. {\em Subelliptic eigenvalue problems.} Conference on harmonic analysis in honor of Antoni Zygmund, Vol. II, 590--606, Wadsworth, 1983. 

\bibitem{FSSC96}  Franchi,~B., Serapioni,~R., Serra Cassano,~F. {\em Meyers-Serrin type theorems and relaxation of variational integrals depending on vector fields.} Houston J. Math. {\bf 22} (1996), 859--890. 

\bibitem{FSSC97} Franchi,~B., Serapioni,~R., Serra Cassano,~F. {\em Approximation and Imbedding Theorems for weighted Sobolev spaces associated with Lipschitz continuous vector fields.}, Bollettino U.M.I., 7, {\bf 11-B} (1997), 83--117.

\bibitem{GM10}  Garofalo,~N., Marola,~N. {\em Sharp capacitary estimates for rings in metric spaces.} Houston J. Math. {\bf 36} (2010), 681--695.

\bibitem{GN98} Garofalo,~N., Nhieu,~D.~M. {\em Lipschitz continuity, global smooth approximations and extension theorems for Sobolev functions in Carnot-Caratheodory spaces}. J. Anal. Math. {\bf 74} (1998), 67--97.

\bibitem{GY99} Grigor'yan,~A., Yau,~S.-T. {\em Decomposition of a metric space by capacitors.} ``Differential equations: La Pietra 1996'', Ed. Giaquinta et. al., Proceedings of Symposia in Pure Mathematics, {\bf 65}, 1999, 39--75.

\bibitem{GNY} Grigor'yan,~A., Netrusov,~Y., Yau~S.-T. {\em Eigenvalues of elliptic operators and geometric applications.} Surveys in differential geometry. Vol.~IX, 147--217. Surv. Diff. Geom., IX, Int. Press, Somerville, MA, 2004.

\bibitem{Ko} Kokarev,~G. {\em Variational aspects of Laplace eigenvalues on Riemannian surfaces.} \\ arXiv:1103.2448v2.

\bibitem {Korv} Korevaar,~N. {\em Upper bounds for eigenvalues of conformal metrics} J. Diff. Geom., {\bf 37} (1993), 73--93.

\bibitem{Matt} Mattila,~P. {\em Geometry of sets and measures in Euclidean spaces. Fractals and rectifiability.} Cambridge Studies in Advanced Mathematics, {\bf 44.} Cambridge University Press, 1995. xii+343 pp. 

\bibitem{MS} Menikoff,~A., Sj\"ostrand,~J. {\em On the eigenvalues of a class of hypoelliptic operators.} Math. Ann., {\bf 235} (1978), 55--85.

\bibitem{Mitch} Mitchell,~J. {\em On Carnot-Carath\'eodory metrics.} J. Diff. Geom. {\bf 21} (1985), 35--45.

\bibitem{NSW} Nagel,~M., Stein,~E.~M., Wainger,~S. {\em Balls and metrics defined by vector fields I: basic properties.} Acta Math. {\bf 155} (1985), 103--147. 


\bibitem{Po} Ponge,~R. {\em Heisenberg calculus and spectral theory of hypoelliptic operators on Heisenberg manifolds.} Mem. Amer. Math. Soc. {\bf 194} (2008), no. 906, viii+134 pp.

\bibitem{Sj} Sj\"ostrand,~J. {\em Eigenvalues of hypoelliptic operators and related methods.} Proc. ICM (Helsinki, 1978), 797--801, Acad. Sci. Fennica, Helsinki, 1980.

\end{thebibliography}
}
\end{document}